\documentclass[12pt]{amsart}
\usepackage{geometry}       
\usepackage{graphicx}
\usepackage{amsmath,amscd,amssymb,dsfont,stmaryrd}
\usepackage{amssymb}
\usepackage{dsfont}
\usepackage{tikz-cd}
\usepackage{epstopdf}
\usepackage{mathrsfs}
\usepackage{hyperref}
\usepackage{enumerate}
\usepackage{graphicx}
\usepackage{float}
\usepackage{tikz}
\usetikzlibrary{matrix,arrows}
\newcommand{\act}{\curvearrowright}
\newcommand{\bb}[1]{\mathbb{#1}}

\newcommand{\cal}[1]{\mathcal{#1}}

\newcommand{\raw}{\rightarrow}

\newcommand{\inv}{^{-1}}
\DeclareMathOperator{\Hom}{Hom}
\DeclareMathOperator{\GL}{GL}
\DeclareMathOperator{\SL}{SL}

\DeclareMathOperator{\sln}{\SL_n\!\bb{C}}

\DeclareMathOperator{\SU}{SU}

\DeclareMathOperator{\Id}{Id}

\newcommand\C{\bb{C}}
\newcommand\G{\Gamma}

\newcommand\HGG{\Hom(\kern 0.05em\G,G\kern 0.05em)}
\newcommand\HGK{\Hom(\G,K)}
\newcommand\HKN{\Hom_\textrm{KN}(\G,G)}

\newcommand\rhob{\rho\left(\left< b \right>\right)}
\newcommand\ctF{\cal{\tilde{F}}}
\newcommand\ctFK{\ctF_K}
\newcommand\embed{\hookrightarrow}
\newcommand\abs[1]{\left\vert {#1} \right\vert}
\newcommand\norm[1]{\left\Vert {#1} \right\Vert}

\newtheorem*{namedtheorem}{\theoremname}
\newcommand{\theoremname}{testing}

\theoremstyle{plain}
\newtheorem{thm}{Theorem}[section]
\newtheorem{lem}[thm]{Lemma}
\newtheorem{prop}[thm]{Proposition}
\newtheorem{cor}[thm]{Corollary}

\theoremstyle{definition}

\title{The topology of Baumslag-Solitar representations}
\author{Maxime Bergeron}
\address{Department of Mathematics, University of Chicago, 5734 S. University Avenue, Chicago IL, 60637, United States of America}
\email{mbergeron@math.uchicago.edu}
\urladdr{http://http://www.math.uchicago.edu/~mbergeron/}
\thanks{M.B.\ was supported by an NSERC postdoctoral fellowship and NSF Award No. DMS-1704692.}
\author{Lior Silberman}
\address{Department of Mathematics, The University of British Columbia, Room 121 - 1984 Mathematics Road, Vancouver BC, V6T 1Z2, Canada}
\email{lior@math.ubc.ca}
\urladdr{http://www.math.ubc.ca/~lior/}
\thanks{L.S.\ was partly supported by an NSERC Discovery Grant}

\begin{document}

\begin{abstract}
Let $\Gamma=\langle a,b \,|\, a b^{\,p} a\inv = b^{\,q}\rangle$ be a
Baumslag--Solitar group and $G$ be a complex reductive algebraic group with 
maximal compact subgroup $K<G$.  We show that, when $p$ and $q$ are relatively prime
with distinct absolute values, there is a strong deformation retraction
retraction of $\HGG$ onto $\HGK$.
\end{abstract}

%\keywords{Representation variety, Baumslag-Solitar,  }
\maketitle

\section{Introduction}
Let $\G$ be a finitely generated group (fix a finite generating set
$S\subset \Gamma$) and let $G$ be a (topological) group.
Identifying a group homomorphism $\rho\colon\G\to G$ with its restriction
$\rho\vert_S$ identifies $\HGG$ with the subset of $G^S$ cut out by the
relations defining $\G$.  In particular, $\HGG$ is a closed subset and we
equip it with the relative topology.  In the case under consideration, $G$
will be a complex linear algebraic group and this identification also endows
$\HGG$ with the structure of a complex affine algebraic variety.

For any topological subgroup $K<G$, the inclusion of $K$ in $G$ induces
a topological inclusion $\HGK\embed\HGG$.  When $K$ is compact, 
$\HGK$ is itself compact and thus more amenable to topological analysis.
In particular, individual elements of $\HGK$ are also often easier to analyze. For example, every unitary representation of $\G$ is completely reducible whereas
this is usually not the case for general linear representations.
For this reason, it is remarkable that in some cases
the inclusion above is in fact a homotopy equivalence, due to the existence
of a strong retraction from $\HGG$ onto $\HGK$. 
This should be contrasted with the fact that, for general finitely generated groups,  rigidity results (e.g. those of Selberg \cite{selberg1963discontinuous}) ensure that this inclusion is not even a bijection of connected components.

    Nevertheless, recent success in producing such retractions has been achieved using a mixture of topological and algebraic tools for classical families of finitely generated groups. Notable positive results include the case of  free-abelian groups  by Pettet--Souto \cite{pettet2013commuting},  torsion-free expanding nilpotent groups by Silberman--Souto \cite{silbermansouto} and general nilpotent groups  by Bergeron \cite{bergeron2015topology}. 

In this paper, we turn our attention to a class of groups that has often served as a testing ground for ideas in geometric group theory: the Baumslag-Solitar groups 
$$
\mathrm{BS}(p,q):=\langle a,b \,|\, a b^{\,p} a\inv = b^{\,q}\rangle.
$$
It is, in some sense, their ``bad behaviour'' that motivated their discovery in the work of  Baumslag--Solitar \cite{baumslag1962some} as the first examples of non-Hopfian residually finite groups. More recently, they also proved to be interesting cases in   Gromov's proposed  classification of finitely generated groups up to quasi-isometry as exhibited by  Farb--Mosher \cite{farb1999quasi,farb1996rigidity} and Whyte \cite{whyte2001large}. 

From our point of view,  analyzing the topology of their representation varieties, we think of the Baumslag-Solitar groups as the simplest interesting infinite groups which are not lattices in Lie groups. Our main result establishes the following:

\begin{thm}
	Let $p$ and $q$ be relatively prime integers with distinct absolute values and consider the Baumslag-Solitar group $$\Gamma=\langle a,b \,|\, a b^{\,p} a\inv = b^{\,q}\rangle.$$ 
	If $G$ is the group of complex points of a reductive algebraic group and $K < G$ is a maximal compact subgroup, then there is a strong deformation retraction of $\HGG$ onto $\HGK$. In particular, the two spaces are homotopy equivalent.
\end{thm}

We conclude this introduction with a short outline of the proof:
\begin{enumerate}
\item An application of Kempf--Ness theory provides a retraction of $\HGG$ to a subset
      $\HKN$ which, in particular, consists of completely reducible
      representations.
\item We  produce an integer $O$ (depending only on $\G$ and $G$) such that, for every
      representation $\rho\in\HKN$, the abelian group $\rhob$ is finite and of order at most $O$.
\item Letting $\ctF$ (respectively \ $\ctFK$) denote the manifold of subgroups of $G$
      (respectively of $K$) of order at most $O$, we show  that 
$$
\HKN\raw\cal{\tilde{F}},\,\,\,\rho\mapsto\rho(\langle b\rangle)
$$
      is a fiber bundle onto a union of connected components.
      Since the inclusion of $\tilde{\cal{F}}_K$ into $\tilde{\cal{F}}$ is
      a homotopy equivalence, it follows that $\HKN$ retracts 
      onto the subset consisting of those $\rho$ with
      $\rhob\subset K$.
\item Each fiber of the bundle is the domain of  a natural covering map onto a subset of the semisimple elements of a reductive group $G'$ which admits a deformation retraction into a compact
      subgroup $K'\subset G'$. Lifting this retraction to the fibers, we obtain the desired retraction of $\HKN$ onto $\Hom(\Gamma,K)$.

\end{enumerate}
\textbf{Acknowledgements.}
The first named author would like to thank Juan Souto for enlightening discussions about the topology of representation spaces and Benson Farb for inspiring discussions about Baumslag-Solitar groups. 

\section{Background on Reductive Algebraic Groups}

An (affine) \emph{algebraic group} is, properly, an (affine) variety endowed
with a group law for which the group operations are morphisms of varieties,
i.e., polynomial maps.  For our purposes, we will use the terms ``affine variety"
to mean an affine algebraic set (the zero locus of a family of complex
polynomials in several variables) and  ``algebraic group" to mean the 
set of complex points of the affine algebraic group.
It turns out that affine algebraic groups are \emph{linear} so we may (and
always will) identify them with a Zariski closed subgroup of $\sln$
(see, for instance, Springer \cite{Springer:LinearAlgGps}).  

Given a compact Lie group $K$, the Peter-Weyl Theorem provides a
faithful embedding $K\hookrightarrow \GL_n\!\bb{R}$ for some $n$.
Identifying $K$ with its image realizes it as a real algebraic subgroup
of $\GL_n\!\bb{R}$. We then define the \emph{complexification} $G:=K_\bb{C}$ to be the vanishing locus in $\GL_n\!\bb{C}$ of the ideal defining $K$. The group $G$ is a complex algebraic group which is independent  up to isomorphism of the embedding provided by the Peter-Weyl Theorem. A complex linear algebraic group $G$ is \emph{reductive} if and only if it is the complexification of a compact Lie group $K$ as above. In this case, $K$ is always a maximal compact subgroup $G$. We summarize some well-known facts in the following (see, for instance, Helgarson \cite{helgason2001differential}):

\begin{prop}\label{symmetric space}
Let $\cal{S}(G)$ denote the set of maximal compact subgroups of $G$.
This set admits the structure of a connected non-positively curves symmetric
space in such a way that the action
$$
G\act \cal{S}(G),\,\,\,(g,K)\mapsto gKg\inv
$$
is a smooth isometric action. In particular:
\begin{enumerate}
	\item any two maximal compact subgroups of $G$ are conjugated by an element in the identity component $G^\circ$, and
	\item if a compact subgroup $H\subset G$ normalizes a maximal compact subgroup\\ $K\in\cal{S}(G)$, then $H\subset K$. \qed
\end{enumerate}
	\end{prop}
Let us denote the centralizer and normalizer of a subgroup $H\subset G$ by 
$$N_G(H):=\{g\in G\,|\,gHg\inv=H\},\text{ and }Z_G(H):=\{g\in G\,|\,\forall h\in H:ghg\inv=h\}.$$ We will also make use of the subset $N_{G^\circ}(H):=G^\circ\cap N_G(H)$.  
We record the following possibly non-standard facts for later use:

\begin{lem}[Pettet-Souto \cite{pettet2013commuting}, Lemma 4.4] \label{normalizer homotopy}
	Let $G$ and $K$ be as above.  If $H\subset K$ is a subgroup, then  the natural   inclusion between normalizers $$N_K(H)\hookrightarrow N_G(H)$$ $$N_{K^\circ}(H)\hookrightarrow N_{G^\circ}(H)$$ are homotopy equivalences.\qed
\end{lem}

\begin{cor} \label{maximal compact normalizer}
	For every abelian $A\subset K$,  $N_K(A)$ is  maximal compact  in  $N_G(A)$.
\end{cor}
\begin{proof}
	We first observe that $N_G(A)$ is reductive. Indeed, since $A$ is diagonalizable,  the identity component of the normalizer $N_G(A)^\circ$ coincides with the identity component of the centralizer $Z_G(A)^\circ$ which is itself reductive by a result of Richardson \cite[Theorem 6.1]{richardson1988conjugacy}.
	
	Now, the compact subgroup $N_K(A)\subset N_G(A)$ is contained in some maximal compact subgroup $\kappa\subset N_G(A)$. Here, it follows from Lemma \ref{normalizer homotopy} and the fact that $N_G(A)$ is reductive that the inclusions $N_K(A)\hookrightarrow N_G(A)$ and $\kappa\hookrightarrow N_G(A)$ are homotopy equivalences. This implies that the inclusion of $N_K(A)$ into $\kappa$ is also a homotopy equivalence.   Since the latter two are closed manifolds, this implies that the inclusion is surjective, meaning that $\kappa= N_K(A)$ after all. \end{proof}

Finally, let $V$ be a complex vector space and consider a $G$-variety
$X\subset V$, i.e., suppose that $G$ acts on $X\subset V$ via an
embedding $G\subset \GL(V)$. If $K\subset G$ is a maximal compact subgroup,
we may assume without loss of generality that $V$ is equipped with a
$K$-invariant Hermitian norm $\norm{\cdot}$.  The restriction of this norm
to $X$ provides a gateway to analyze its topology through a package of
results known as \emph{Kempf--Ness Theory}
(see \cite[Section 3]{bergeron2015topology} for
an introduction and \cite[Chapter 5]{schwarz2000theorie} for the details).
To state the results we need, write $\mathcal{M}$ for the
\emph{Kempf--Ness Set} of minimal vectors in $X$:
$$\mathcal{M}:=\left\{x\in X: \forall g\in G, \norm{x}\leq\norm{g\cdot x}\right\}\,.$$
Then:
\begin{enumerate}
	\item there is a strong deformation retraction of $X$ onto $\cal{M}$, and
	\item the $G$-orbit of every $x\in\cal{M}$ is closed. 
\end{enumerate}

\section{The Main Theorem}
Let $p$ and $q$ be relatively prime integers with $|p|\neq|q|$. We consider the Baumslag-Solitar group $$\Gamma=\langle a,b \,|\, a b^{\,p} a\inv = b^{\,q}\rangle.$$ Let us also fix once and for all a complex reductive algebraic group $G$ along with a maximal compact subgroup $K\subset G$.

As explained in the introduction, the image of $\HGG$ in $G\times G$ under
the injective map
$$\rho\mapsto\left(\rho(a),\rho(b)\right)$$
is the closed subset cut out by the relation
$$\rho(a)\rho(b)^p\rho(a)^{-1} = \rho(b)^q\,.$$
This simultaneously endows $\HGG$ with compatible structures of an affine
algebraic variety (when $G$ is viewed as an affine algebraic group)
and a Hausdorff topological space (when $G$ is viewed as a Lie group).
It is well known (see, for instance, \cite{lubotzky1985varieties}) and
otherwise easy to see that, in both points of view, the geometric
structures obtained from distinct presentations of the same group $\Gamma$
are isomorphic. The choice of presentation is thus immaterial.

If $K\subset G$ is a maximal compact subgroup, the inclusion $K\embed G$
induces is a natural inclusion $\HGK\embed\HGG$ and we would like to show that
there is a deformation retruction of the latter space into the former.

\subsection{Kempf--Ness Theory} \label{applying KN}
In this section we will take the first step of our argument, obtaining
an intermediate deformation retract  between $\HGK$ and $\HGG$.

For this, observe that the action of $G$ on itself by conjugation induces
an action of $G$ on $\HGG$ via
$\left(g\cdot \rho\right)(\gamma) = g\rho(\gamma)g^{-1}$.  In terms of our
embedding this is simply the restriction of the diagonal action of $G$ on
$G\times G$ to the subspace/subvariety $\HGG$.  Fixing an embedding of $G$
in $\sln$, this allows us to view
$$\HGG\subset G^2\subset (M_n\bb{C})^2\cong \bb{C}^{2n^2}$$
as a $G$-variety.
Endowing the ambient vector space with a $K$-invariant norm (for example, we
may assume the embedding maps $K$ to $\SU(n)$ and take the Hilbert--Schmid norm
on $M_n(\C)$), we obtain a retraction of $\HGG$ onto the associated
Kempf--Ness set $\HKN$, given concretely as the set:
$$\left\{\rho\colon\Gamma\to G \mid \forall g\in G:
\norm{g\rho(a)g^{-1}}_\textrm{HS}^2+ \norm{g\rho(b)g^{-1}}_\textrm{HS}^2
\geq \norm{\rho(a)}_\textrm{HS}^2+ \norm{\rho(b)}_\textrm{HS}^2 \right\}\,.$$

\subsection{Manifolds of Finite Abelian Groups}\label{manifolds of finite abelian groups}
Keeping the notation as above, let $B$ denote the subgroup of $\Gamma$
generated by $b$. Our goal in this section is to realize the Kempf--Ness set
$\HKN$ as the total space of a bundle over a manifold of
finite abelian subgroups of $G$ (see also Bergeron-Silberman
 \cite[Proposition 2.1]{bergeron2016note} and Pettet-Souto \cite[Section 3.4]{pettet2013commuting}). 
\begin{prop} \label{bundle}
	The set of abelian groups
$$\cal{F}:=\left\{\rho(B)\subset G:\rho\in \HKN\right\}$$
(clearly invariant under conjugation by $G$) admits a manifold structure with
finitely many connected components such that the $G$-action is smooth and
 the projection map
	$$
	\pi:\HKN\raw\cal{F},\quad \pi(\rho):=\rho(B)
	$$
is a locally trivial fiber bundle. Moreover, if $H\in\cal{F}$, the image of a representation $\rho\in\pi\inv(H)$ is contained in the normalizer $N_G(H)$.
\end{prop}

The following two Lemmas and their proofs are inspired by the work
of McLaury \cite{mclaury2012irreducible} on irreducible linear representations
of Baumslag--Solitar groups.  We quote Theorem 3.2 of his paper as
Lemma \ref{lem:rhoGamma-solvable} below but give a full proof since there
are minor errors in the argument given in \cite{mclaury2012irreducible}.
Lemma \ref{uniform bound} then generalizes Theorems 3.3 and 5.1 of  \cite{mclaury2012irreducible} 
to our context of representations in general reductive groups lying
in the Kempf--Ness set.  These Lemmas are the only places in the paper
where we directly use the hypothesis that $p$ and $q$ are relatively prime.

\begin{lem}\label{lem:rhoGamma-solvable}
Let $G$ be a linear algebraic group and let $\Gamma$
be a Baumslag--Solitar group with $p$ and $q$ relatively prime.
If $\rho\in\HGG$, then $\rho(\Gamma)$ is solvable.
\end{lem}
\begin{proof}
	Let $L = \overline{\rho(\Gamma)}$ be the Zariski closure of the
image of $\Gamma$, and let $H = \overline{\rho(B)}$ be the Zariski closure
of $\rho(B)$.  The subgroups $H_1 = \overline{\left<\rho(b^p)\right>}$
and $H_2 = \overline{\left<\rho(b^q)\right>}$ are of finite index in $H$
(the indices divide $p,q$ respectively), and in particular we have
an identity $H^\circ = H_1^\circ = H_2^\circ$ of (algebraic) connected
components.  Since $\rho(a)$ conjugates $\rho(b^p)$ to $\rho(b^q)$
it also conjugates $H_1$ to $H_2$ and therefore normalizes $H^\circ$.
$H^\circ$ is also normalized by $\rho(b)\in H$ and we conclude that $H^\circ$
is normalized by $\rho(\Gamma)$ and is hence normal in $L$.

	The two finite subgroups $H^1/H^\circ, H^2/H^\circ < L/H^\circ$ are
conjguate by the image of $\rho(a)$ in $L/H^\circ$.   In particular,
they have the same order.  Now these two subgroups are contained in the
finite group $H/H^\circ$, so they also have the same index there -- in 
other words $[H:H_1] = [H:H_2]$.  However, as noted above these indices
divide $p,q$ respectively.  Since $\mathrm{gcd}(p,q) = 1$ we conclude
that $H = H_1 = H_2$ and thence that $H$ itself is normalized by $\rho(a)$.
It also contains $\rho(b)$ and it now follows that $H$ is normal in $L$.

	Finally, since $\rho(b)\in H$ the image of $\rho(a)$ in $L/H$
is a (Zariski-) topological generator.  Having shown that both $H$ and $L/H$
are abelian we conclude that $L$ and its subgroup $\rho(\Gamma)$ are
solvable.
\end{proof}

\begin{lem}\label{uniform bound}
	If $\rho\in\HKN$, then $\rho(\G)$ is virtually abelian and
consists of semisimple elements.  Furthermore, the abelian group $\rho(B)$
is normal in $\rho(\G)$, finite, and its order is bounded by a constant
$O=O(\G,G)$ that is independent of $\rho$.
\end{lem}
\begin{proof}
	Recall that, since $\rho$ lies in the Kempf--Ness set, its
orbit under conjugation must be closed.  A result of
Richardson \cite[Theorem 3.6]{richardson1988conjugacy} now implies that the
Zariski closure $L$ of the image of $\rho$ is reductive.  It follows that
$L^\circ$ is a connected reductive solvable group, in other words an algebraic
torus (possibly zero-dimensional).  This makes $L$ and its subgroup $\rho(\G)$
virtually abelian and ensures that every element of $L^\circ$ is semisimple.
Since every element of $L$ has a positive power in $L^\circ$, the
Jordan decomposition shows that every element of $L$ (in particular, 
every element of $\rho(\Gamma)$) is semisimple as well.

	Now $H =\overline{\rho(B)}$ is a commutative subgroup consisting
of semisimple elements.  By \cite[Cor.\ 3.2.9]{Springer:LinearAlgGps}
the group $N_G(H)/Z_G(H)$ is finite, and it follows that conjguation
by $\rho(a)$ is an automorphism of $H$ of finite order.  In particular,
there is $r$ such that $\rho(a^r)$ commutes with $H$, in particular with
$\rho(b)$.  On the other hand,
$a b^p a^{-1} = b^q$ implies $a b^k a^{-1} = b^{qk/p}$
the group $H/Z_G(L)$ is finite, and it follows that conjugation
by $\rho(a)$ is an automorphism of $L$ of finite order.  In particular,
there is an $r$ such that $\rho(a^r)$ commutes with $L$, hence with $\rho(b)$.
On the other hand, $a b^p a^{-1} = b^q$ implies $a b^k a^{-1} = b^{qk/p}$
whenever $k$ is divisible by $p$.  Conjugating $b^{p^r}$ by $a$ $r$ times
we obtain the identity $a^r b^{p^r} a^{-r} = b^{q^r}$, and upon applying
$\rho$ that $\rho(b)^{p^r} = \rho(b)^{q^r}$.  Since $p$ and $q$ are distinct,
we conclude that $\rho(b)$ has finite order (dividing $\abs{p^r-q^r}$)
and that $\rho(B)$ is finite.

	Since $\rho(B)$ is finite, is is also Zariski closed and
$H = \rho(B)$.  It follows that $\rho(B)$ is normalized by $L$, hence by
$\rho(\Gamma)$.  As before, $\rho(\Gamma)/\rho(B)$ is cyclic, generated
by the image of $a$.

	We now rely on our linear embedding $G\subset\sln$.  Since $L$
is reductive, this linear representation of $L$ decomposes as a direct
sum of irreducible representations.  Since $\rho(\Gamma)$ is Zariski-dense
in $L$, those subspaces are also irreducible as representations of $\Gamma$.
The calculations of \cite[Thm.\ 5.1]{mclaury2012irreducible} now show
that every eigenvalue of $b$ in this representation is a root of unity of
order dividing $\abs{p^k-q^k}$ for some $1\leq k\leq n$, hence dividing
$O = \prod_{k=1}^{n}\abs{p^k-q^k}$.  Since $\rho(b)$ is semisimple it follows
that $\rho(b)$ has order dividing $O$, and in particular
that $\abs{\rho(B)}\leq O$.
	\end{proof}

\begin{proof}[Proof of Proposition \ref{bundle}]
Before we study $\cal{F}$, we first consider the slightly larger set
$$
\ctF:=\left\{H<G:H\text{ is abelian of order at most }O\right\}.
$$
Observe that $G^{\circ}$ (the identity component of $G$) acts by conjugation on $\tilde{\cal{F}}$ with closed stabilizers.  As such, we can endow
$\tilde{\mathcal{F}}$ with the orbifold structure with respect to which each
$G^{\circ}$-orbit is a connected homogeneous $G^{\circ}$-manifold.
Concretely, if we define the ``connected normalizer'' as
$N_{G^{\circ}}(H):=N_{G}(H)\cap G^{\circ},$
then the connected component of $H\in\tilde{\mathcal{F}}$ will be its
$G^{\circ}$ orbit, topologized via the identification with
$G^{\circ}/N_{G^{\circ}}(H)$.   There being only finitely many $G$-orbits in
$\ctF$, we nessarily give it the topology of the disjoint union of the orbits.
 
A homomorphism $\rho\colon B\raw G$ need not extend to the full group
$\Gamma$ so the map 
$$\pi:\HKN\raw \ctF,\quad\pi(\rho)=\rho(B)$$
need not be surjective.
Recall, however, that we have defined $\mathcal{F}=\pi\left(\HKN\right)$
and observe that the $G^{\circ}$-equivariance of $\pi$ implies that
$\cal{F}$ is the union of those components of $\ctF$ it intersects,
giving it the desired manifold structure.

For the bundle structure, let $\mathcal{Z}\subset \mathcal{F}$ denote the
connected component of a finite abelian subgroup $H\in \mathcal{F}$, let
$\mathcal{H}:=\pi\inv(\mathcal{Z})\subset \HKN$ and let
$\mathcal{H}_H:=\pi\inv(H)$. We can then identify $\cal{H}$ and the
restriction $\pi$ there with the twisted product 
$$(G^{\circ}\times\mathcal{H}_H)/N_{G^{\circ}}(H)\raw G^{\circ}/N_{G^{\circ}}(H)$$
 where $N_{G^{\circ}}(H)$ acts on $G^{\circ}$ (resp. $\mathcal{H}_H$) by right multiplication (resp. conjugation). This shows that $\pi$ is a locally trivial fibre bundle.	
\end{proof}

We will also need to consider the analogous manifold of abelian subgroups in the maximal compact subgroup $K\subset G$.  To this end, let
$$
\ctFK:=\{H\in\tilde{\cal{F}}\,|\,H\subset K\}
$$
be the manifold of abelian subgroups of $K$ of order at most $O$. It follows, as before, that the conjugation action of $K$ on $\tilde{\cal{F}}_K$ has finitely many orbits and we can endow  $\tilde{\cal{F}}_K$ with the structure of a homogeneous manifold. Concretely, the connected component of $H\in\tilde{\cal{F}}$ is identified with $K^\circ/N_{K^\circ}(H)$. We record the following two facts about manifolds of finite abelian subgroups:
\begin{prop}\label{manifold of subgroups retraction}
	The manifold $\tilde{\mathcal{F}}_K$ is a deformation retract of $\tilde{\mathcal{F}}$.
\end{prop}
\begin{proof}
	The argument is analogous to the one in \cite[Proposition 4.1]{pettet2013commuting}.
We first verify that $\ctFK\embed \ctF$ induces a bijection at the level
of connected components. To see this, recall that any $H\in\cal{F}$, being
compact, can be conjugated into $K$ by an element of $G^\circ$, which shows
that $\pi_0(\ctFK)\raw \pi_0(\ctF)$ is surjective.  Conversely,
\cite[Lemma 3.4]{pettet2013commuting} asserts that two subgroups of $K$
which are conjugate by an element $G^\circ$ are also conjugate by an element
of $K^\circ$, in other words that our map of connected components is injective.
	
Let $\cal{Z}$ denote a component of $\tilde{\cal{F}}$ and let $\cal{Z}_K$ denote the unique component of  $\tilde{\cal{F}}_K$ contained in $\cal{Z}$. To complete the proof, it suffices to show that $\cal{Z}_K\hookrightarrow \cal{Z}$ is a homotopy equivalence. Fix a subgroup $H\in \cal{Z}_K$ and use it to identify $\cal{Z}_K$ and $\cal{Z}$ with $K^\circ/N_{K^\circ}(H)$ and $G^\circ/N_{G^\circ}(H)$ respectively. We now have a commutative diagram
$$\begin{tikzcd}
  N_{K^\circ}(H) \arrow[r,] \arrow[d,]
    & K^\circ \arrow[d,] \arrow[r,]  & \cal{Z}_K \arrow[d,] \\
  N_{G^\circ}(H) \arrow[r,] & G^\circ \arrow[r,]
& \cal{Z} \end{tikzcd}$$
where the two rows are principal bundles. Since the first two columns are
homotopy equivalences (that was Lemma \ref{normalizer homotopy}),
it follows that the last column is also a homotopy equivalence. 
	\end{proof}

\begin{prop}[Cartan's Lemma]\label{compact choice}
	There is a continuous map $\kappa\colon\ctF\raw \mathcal{S}(G)$ with:
	\begin{enumerate}
		\item $H\subset \kappa(H)$ for all $H\in\ctF$, and
		\item $\kappa(H)=K$ for all $H\in\ctFK$.
	\end{enumerate}
\end{prop}
\begin{proof}
	Let $d(\cdot,\cdot)$ denote the Riemannian metric on the simply
connected non-positively curved complete metric space $\cal{S}(G)$.
The strict convexity of $d(\cdot,\cdot)^2$ implies that the function
	$$
	f_H\colon\cal{S}(G)\raw\bb{R}_+,\quad f_H(\kappa)=\sum_{g\in H}d(\kappa,gKg\inv)^2
	$$
	is strictly convex. Since it is also proper, the function $f_H$ attains its minimum at a unique point $\kappa = \kappa(H)\in\cal{S}(G)$. It follows from the $H$-invariance of $f_H$ that this minimum is also $H$-invariant and,
hence, $H$ is contained in the maximal compact subgroup stabilizing $\kappa(H)$. Notice also that if $H\subset K$, then the non-negative function $f_H$ vanishes at $K$, showing that $\kappa(H)=K$. Finally, it follows from the non-positive
curvature that the minimum of $f_H$ depends continuously on $H$, that is, that
	$$
	\kappa:\tilde{\cal{F}}\raw\cal{S}(G),\quad H\mapsto \kappa(H)
	$$
	is continuous.
\end{proof}

\subsection{The Covering Map}Before going further, we need another bundle. 
For any finite abelian subgroup $H\subset G$ let $$G_H:=N_G(H)/H$$ and let us denote the corresponding natural projection by  $$p:N_G(H)\raw G_H.$$ Notice, as in the proof of Corollary \ref{maximal compact normalizer}, that $N_G(H)$ and thus $G_H$ are reductive.

Recall from the proof of Proposition \ref{bundle} that if $\rho\in\pi\inv(H)$ then $\rho(B)=H$ and $\rho(\Gamma)\subset N_G(H)$. Thus, given any such $\rho$ we obtain a representation  
$$
p\circ \rho: \Gamma\raw G_H
$$ with image landing in the subset $(G_H)_s$ of semisimple elements of $G_H$.
  By construction, the subgroup $B$ of $\Gamma$ is contained in the kernel of $p\circ \rho$ so this map is uniquely determined by the image of the element $a\in \Gamma$. Altogether, we have a map
 $$
 p_\ast: \pi\inv(H)\raw (G_H)_s
 $$
 with finite fibers because the projection $N_G(H)\raw G_H$ has finite kernel. In particular, $p_\ast(\pi\inv(H))$ is contained in the set $(G_H)_{s}$ of semisimple elements of $G_H$.

  The action by conjugation of $N_{G^\circ}(H)$ on $N_G(H)$ induces an action
$$
(G_H)_s \curvearrowleft N_{G^\circ}(H),\quad (p(g),h)\mapsto p(hgh\inv).
$$
%the elements of the normalizer we are working with actually arise from representations
This  can now be used to construct the twisted product 
 $$
 \mathcal{G}:=(G^\circ\times ({G}_H)_{s})/N_{G^\circ}(H)\xrightarrow{\pi_\cal{G}} G^\circ/N_{G^\circ}(H)=\mathcal{Z}.
 $$
Here, as in the previous section, $\mathcal{Z}$ denotes the component of $\mathcal{F}$ containing $H$.
The map $p_\ast$ is  $N_{G^\circ}(H)$-equivariant and, in particular, it induces a morphism of fiber bundles from $\mathcal{H}$ to $\mathcal{G}$ where $\mathcal{H}=\pi\inv(\mathcal{Z})$ is a component of the bundle defined in Proposition \ref{bundle}. In fact, can say a bit more: %[c.f. Goldman \cite{goldman1988topological} Lemma 2.2]

\begin{prop}\label{covering map}
	The bundle map $p_\ast:\mathcal{H}\raw \mathcal {G}$ is a finite cover onto a union of connected components of $\mathcal{G}$.
\end{prop}
\begin{proof}
Since $p_\ast$ is a morphism of fiber bundles, it suffices to show that the map is a covering onto a union of connected components at the level of fibers. In other words, for every $H\in\cal{Z}$, we need to verify our claim for the map 
$$
p_\ast:\pi\inv(H)\raw\pi_\cal{G}\inv(H)
$$
where $\pi\inv(H)$ consists of homomorphisms $\rho:\Gamma\raw N_G(H)_s$ with $\rho(B)=H$ and $\pi_\cal{G}\inv(H)=(G_H)_s=(N_G(H)/H)_s$. Since all the objects in the game are semi-analytic sets, it suffices to show that $p_\ast$ is locally injective and has the path-lifting property.

We start by showing that $p_\ast$ is locally injective. Suppose that $\rho$ and $\psi$ are representations in $\pi\inv(H)$ that are very close to each other.  Recalling that we identify representations with the image of our generating set for $\Gamma=\langle a,b \rangle$, this means that the pair of elements $\rho(a)$ and $\psi(a)$ (respectively $\rho(b)$ and $\psi(b)$) differ by an element $g_a$ (respectively $g_b$) of $N_G(H)$ very close to the identity. %actually they must lie in the identity component which coincides with the identity component of the centralizer of H
Now, if  $p_\ast(\rho)=p_\ast(\psi)$ then $g_a$ and $g_b$ are elements of $H$. Moreover, since $H$ is discrete and $g_a$ and $g_b$ are very close to the identity, it follows that they must both be the identity. This shows that $\rho=\psi$ and thus $p_\ast$ is locally injective.

We show that $p_\ast$ has the path-lifting property. To do this, it is convenient to identify $\pi_\cal{G}\inv(H)$ with the subset $(G_H)_s\times \Id\subset G_H\times G_H$ and to view $\pi\inv(H)$ as a subset of $N_G(H)\times N_G(H)$. 
Let $\rho\in\pi\inv(H)$ and consider a  path 
$$
\gamma:[0,1]\raw (G_H)_s\times \Id\subset G_H\times G_H,\,\,\,\,\gamma(t)=(\gamma_1(t),\gamma_2(t))
$$
with $\gamma(0)=p_\ast(\rho)$.  
This path always admits  a unique lift $\tilde\gamma$ to the finite cover 
$$N_G(H)\times N_G(H)\raw G_H\times G_H$$
with $
\tilde{\gamma}(0)=(\rho(a),\rho(b)).
$
 To complete the proof we need to show that 
$$
\tilde{\gamma}(t)\in\pi\inv(H)\subset N_G(H)\times N_G(H), 
$$
i.e., we need to check that the relations of $\Gamma$ are preserved within the path and that the corresponding homomorphisms map $B$ onto $H$.

Let $w$ be a word in the letters $a$ and $b$ which is trivial in $\Gamma$. Substituting $\gamma_i(t)$ (respectively $\tilde\gamma_i(t)$) for $a$ and $b$ in $w$ yields a continuous path $w(t)$ in $G_H$  (respectively $N_G(H)$) such that $w(t)$ is the image of $\tilde{w}(t)$ under the quotient map. By construction, $w(t)$ is the identity element of $G_H$ for all $t$. %it is trivial in $\Gamma$ so it is trivial in $\rho(\Gamma)$ so it is trivial in $\rho(\Gamma)/\rho(B)$
It follows that $\tilde{w}(t)\in H$ for every $t$ and, since $H$ is discrete, this is a constant map as well. Given that $\tilde{w}(0)$ is trivial (it corresponds to the homomorphism $\rho$), we conclude that $\tilde{w}(t)$ is trivial for all $t$ as well. This shows that $\tilde\gamma(t)$ is a path of homomorphisms $\rho_t:\Gamma\raw N_G(H)_s$.  The image of these homomorphisms is always contained in $N_G(H)_s$  because $\gamma_i(t)$ is semisimple and the quotient map is finite-to-one.

Finally, we show that the homomorphisms $\rho_t$ determined by $\tilde{\gamma}(t)$ send the subgroup $\langle b \rangle = B\subset \Gamma$ onto $H$. This follows, once again, because the homomorphisms to $G_H$ determined by $\gamma(t)$ map $b$ to the identity in $G_H$ and, as above, this implies that $\rho_t(b)$ is a fixed element of $H$. Since $\rho_0(b)$ is a generator of $H$ this must then hold for  $\rho_t(b)$.
\end{proof}

\subsection{Putting the Puzzle Back Together}
  
 Recall that we denote the manifold of maximal compact subgroups of $G$ by $\mathcal{S}(G)$ and, from Proposition \ref{compact choice}, that we have a continuous map $\kappa:{\cal{F}}\raw \cal{S}(G)$. We use this map to define sub-bundles of $\pi:\Hom_{KN}(\Gamma,G)\raw \cal{F}$ and $\pi_{\cal{G}}:\cal{G}\raw\cal{Z}$  which will be intermediate steps in our deformations of $\Hom_{KN}(\Gamma,G)$ onto $\Hom(\Gamma,K)$.
 
 First, consider the set
 $$
 \cal{K}:=\{\rho\in\HKN\,|\, \rho(\Gamma)\subset \kappa(\pi(\rho))\}.
 $$
Observe that, by the continuity of $\kappa$, the restriction of $\pi$ to $\cal{K}$ is a sub-bundle $\pi_\cal{K}:\cal{K}\raw\cal{F}$. Since $\kappa(H)=K$ for every $H\in \cal{F}_K$, it follows that 
$$\Hom(\Gamma,K)=\pi_\cal{K}\inv(\cal{F}_K).$$
 Recalling from Proposition \ref{manifold of subgroups retraction} that $\cal{F}_K$ is a deformation retract of $\cal{F}$, it follows moreover that $\Hom(\Gamma,K)$ is a deformation retract of $\cal{K}$.
 
 Second, observe by Corollary \ref{maximal compact normalizer} that  $N_{\kappa(H)}(H)$ is a maximal compact subgroup of  $N_G(H).$ This ensures that  
 $$
 K_H:=N_{\kappa(H)}(H)/H
 $$
is also a maximal compact subgroup of $G_H$. We can now consider the subset $\cal{C}\subset \cal{G}$ defined fiber-wise by specifying that 
$$
\cal{C}\cap \pi_\cal{G}\inv(H)= K_H
$$
for every $H\in\cal{Z}$. It follows, once again, from the continuity of $\kappa$ that $K_H$ varies continuously so we obtain a sub-bundle 
$$
\pi_\cal{C}:\cal{C}\raw\cal{Z}
$$
of $\pi_\cal{G}:\cal{G}\raw\cal{Z}$. In fact, the following holds:
\begin{lem}\label{relationship}Keeping the notation as above, we have 
	$p_\ast\inv(\cal{C})=\cal{K}\cap\pi\inv(\cal{Z})$.
\end{lem}
\begin{proof}
Suppose $\rho:\Gamma\raw G$ is a representation with $p_\ast(\rho)\in\cal{C}$. This means that $$(p\circ\rho)(\Gamma)=p(\rho(\Gamma))\subset K_{\pi(\rho)}.$$ In particular, we see that $\rho(\Gamma)\subset p\inv(K_{\pi(\rho)})=\kappa(\pi(\rho))$ and therefore $\rho\in \cal{K}$.
\end{proof}

In order to alleviate the notation (with the hope that no confusion shall arise), from now on we will  use $\cal{G}$ (respectively $\cal{C}$) to denote the union of the corresponding bundles (respectively sub-bundles) over all connected components of $\cal{F}$ rather than restricting our attention to $\cal{Z}\subset \cal{F}$.

We can then consider the globally induced morphism of fiber bundles 
$$
p_\ast:\HKN\raw \cal{G}
$$
where, applying Lemma \ref{relationship} component by component, we have $p_\ast\inv(\cal{C})=\cal{K}$. Similarly, applying Proposition \ref{covering map} component by component, we have that $p_\ast$ is a covering map onto a union of connected components of $\cal{G}$. To summarize, we have:

\begin{prop}\label{summary proposition}
	The bundle $\pi_\cal{G}:\cal{G}\raw\cal{F}$ contains a sub-bundle $\pi_\cal{C}:\cal{C}\raw\cal{F}$ such that, for every $H\in\cal{F}$, we have:
	\begin{enumerate}
		\item $\pi_\cal{G}\inv(H)=(G_H)_{s},$ and
		\item $\pi_\cal{K}\inv(H)=K_H.$
	\end{enumerate}
	Here, $K_H$ is a maximal compact subgroup of the reductive group $G_H$ and $(G_H)_s$ denotes the set of semisimple elements in $G_H$.
	Moreover, the bundle morphism $$p_\ast:\HKN\raw \cal{G}$$ is a finite cover onto a union of connected components of $\cal{G}$ with $p_\ast\inv(\cal{C})=\cal{K}$.\qed
\end{prop}

\subsection{Proof of the Main Theorem}
With Proposition \ref{summary proposition} at our disposition, we are now ready to complete the proof of our main theorem.  After our applications of Kempf--Ness theory in Section \ref{applying KN}, it suffices to show that $\Hom(\Gamma,K)$ is a deformation retract of $\Hom_{KN}(\Gamma,G)$. 

To begin, recall from Propositions \ref{bundle} and \ref{summary proposition} that we have bundles
$$
\pi:\HKN\raw \cal{F}\text{, and }\pi_\cal{G}:\cal{G}\raw\cal{F},
$$ 
along with respective sub-bundles
$$
\pi:\cal{K}\raw \cal{F}\text{, and }\pi_\cal{C}:\cal{C}\raw\cal{F}.
$$ 
Recall, moreover, that  they are related by a covering map $p_\ast:\HKN\raw \cal{G}$. Our deformation retraction now proceed in three steps:

\emph{Step one}: For every $H\in\cal{F}$, we have a retraction of the fiber $\pi_\cal{G}\inv(H)=(G_H)_s$ onto the fiber $\pi_\cal{C}\inv(H)=K_H$. Since everything in sight is a fiber bundle, this implies that there is a fiber-preserving retraction of $\cal{G}$ onto $\cal{C}$.

\emph{Step two}: Use the covering map $p_\ast:\HKN\raw \cal{G}$ to lift the retraction from step one to a fiber-preserving retraction of $\HKN$ onto the sub-bundle $\cal{K}$.

\emph{Step three}: Lift the deformation retraction of $\cal{F}$ onto $\cal{F}_K$ given by Lemma \ref{manifold of subgroups retraction} to a retraction of $\cal{K}$ onto $\cal{K}\cap\pi\inv(\cal{F}_K)=\Hom(\Gamma,K)$.

This completes the proof of our main theorem.\qed
\bibliographystyle{plain}
\bibliography{bsnote}
\end{document}